\documentclass[12pt]{amsart}

\usepackage{amsmath,amssymb,amsthm,amsfonts}
\usepackage[T1]{fontenc}
\usepackage[utf8]{inputenc}
\usepackage{mathrsfs}

\usepackage[margin=1 in]{geometry}
\usepackage{amscd}
\usepackage{latexsym}
\usepackage{color}
\usepackage{graphicx}
\usepackage{longtable} 
\usepackage{cite} 
\usepackage[colorlinks=true]{hyperref}
\hypersetup{urlcolor=blue, citecolor=blue, linkcolor=blue}

\newtheorem{theorem}{Theorem}[section]

\newtheorem{lemma}[theorem]{Lemma}

\newtheorem{proposition}[theorem]{Proposition}

\newtheorem{remark}[theorem]{Remark}
\newtheorem{ex}[theorem]{Example}

\theoremstyle{definition}
\theoremstyle{remark}
\numberwithin{equation}{section}

\newcommand{\diag}{\operatorname{diag}}

\newcommand{\R}{\mathbb{R}}
\newcommand{\C}{\mathbb{C}}
\renewcommand{\i}{{\mathbf i}}
\newcommand{\T}{\operatorname{UT}} 
\renewcommand{\t}{\mathfrak{ut}} 

\newcommand{\w}{\mathbf{w}}

\newcommand{\y}{\mathbf{y}}

\renewcommand{\b}{\mathbf{b}}

\newcommand{\mtx}[1]{\begin{bmatrix} #1 \end{bmatrix}}

\title{The Karcher Mean on Lie Groups of Upper Triangular Matrices}

\author{Huajun Huang}
\address{Department of Mathematics and Statistics, Auburn University, Auburn, AL 36849, USA}
\curraddr{}
\email{huanghu@auburn.edu}
\thanks{}

\author{Jimmie Lawson}
\address{Department of Mathematics, Louisiana State University, Baton Rouge, LA 70803, USA}
\curraddr{}
\email{lawson@math.lsu.edu}
\thanks{}

\begin{document}

\begin{abstract}
We study the weighted Karcher mean on the Lie group $\operatorname{UT}_n$ of upper triangular matrices with positive diagonal entries. 
Using a divided-difference description of analytic functions on upper triangular matrices and the Parlett recurrence, we prove the existence, 
uniqueness, and analyticity of the weighted Karcher mean and derive recursive formulas for its computation. The theory yields explicit 
expressions in low dimensions and formulas for weighted geometric means on $\operatorname{UT}_n$. More generally, the results reveal a 
close relationship between divided differences, analytic matrix functions, and nonlinear matrix means on triangular Lie groups.
\end{abstract}


\maketitle


\medskip


\section{Introduction}

In this paper we investigate the Karcher mean in the setting  of the  Lie group \(\T_n\) of \(n\times n\) upper triangular
matrices over \(\mathbb{R}\) with positive entries along the diagonal. As is well-known, the exponential
mapping is a diffeomorphism to the Lie group \(\T_n\) from its Lie algebra \(\t_n\) (all upper
triangular real matrices). Thus its inverse map \(\log : \T_n \to \t_n\) is an analytic diffeomorphism.

Given a sequence \(\w=(w_1,\dots,w_m)\) of \(m\) positive real numbers summing to \(1\) and
\(m\) arbitrary but fixed members \(A_1,\dots,A_m\in \T_n\), we define a weighted Karcher mean $X=G(\w ; A_1,\ldots, A_m)$  as a solution \(X\in\T_n\) of the equation
\begin{equation}\label{Karcher equation}
	\sum_{i=1}^m w_i \log(X^{-1}A_i)=0.
\end{equation}
We visualize
\(\w=(w_1,\dots,w_m)\) as weights and \((A_1,\dots,A_m)\) as the corresponding supports
of a finitely supported measure.

In \cite{Law25} the second author used the preceding definition to introduce the weighted Karcher mean on Lie groups.  
(In \cite{Law25} the Karcher mean  is called the exponential mean since its definition is given directly through the exponential 
map from the Lie algebra to the Lie group.)  In an earlier study \cite{Law24} the second author studied the Karcher mean on
unipotent Lie groups and proved the convergence of the projection mean iteration to the Karcher mean in finitely many steps.  
Earlier Pennec and Arsigny developed an alternative approach to the Karcher Lie group mean with what they called  bi-invariant 
means on Lie groups as exponential barycenters associated with the canonical Cartan--Schouten connection, thereby providing 
a rigorous geometric framework for group means \cite{PA12}.  In \cite{PL20} Pennec and Lorenzi extended geometric statistics 
from the Riemannian setting to affine connection spaces; in this framework, they reinterpret Karcher means in terms of exponential barycenters.

In a forerunner for this paper \cite{Hu26},  the first author has treated the case of the Karcher mean for \(\T_2\) and given a complete characterization 
of the convergence behavior of the projection mean iteration in this setting.  In this paper we show in the general setting of $\T_n$ that the Karcher
mean uniquely exists and provide computational tools for computing it.  These tools are demonstrated by working our formulas for the Karcher mean
for $\T_3$ and $\T_4$.  We show, among other things, that the Karcher mean always uniquely exists and is analytic.

Motivation for the study of the Karcher  mean is its appearance and importance
in the area of manifold-valued data analysis, where it can serve as a mean for data points, which can often be conceived
as elements of a Lie group. In \cite{PL20} the authors illustrate this approach in the area
of medical image analysis.

\section{Preliminaries} 
We adopt the following notation:
\begin{itemize}
	
	\item Let \(\diag(A)\) denote the diagonal matrix formed from the diagonal entries of \(A\), and let \(\diag(a_1,\ldots,a_n)\) denote the diagonal matrix with diagonal entries \(a_1,\ldots,a_n\) in that order.
	
	\item For \(i,j \in \{1,\ldots,n\}\), let \(E_{ij}\) denote the \(n\times n\) matrix whose only nonzero entry is \(1\) in the \((i,j)\)-entry.
	
	\item For \(i,j \in \{1,\ldots,n\}\) with \(i\le j\) and $A\in \t_n$, let \(A[i:j]\) denote the principal submatrix of \(A\) whose rows and columns are indexed by \(\{i,i+1,\ldots,j\}\).
	          We frequently refer to these principal submatrices as diagonal blocks.  The smallest diagonal blocks are the $1\times 1$ submatrices with entry a 
	          diagonal element.  

\end{itemize}

\subsection{Subgroups of Diagonal Blocks}
The set of diagonal blocks of size $[i:j]$ in $\T_n$ form a group isomorphic to $\T_k$ under matrix multiplication, where $k=j-i+1$.  This group isomorphically embeds
in $\T_n$  by extending a diagonal block $A[i:j]$ by sending an entry outside the diagonal block to $1$ if it is on the diagonal and $0$ otherwise.  It is  easily checked that this is
a group embedding.  We identify the groups via this isomorphic embedding and work in whatever setting is conventient.

Matrix sums, scalar products, and matrix products  in \(\t_n\) are  preserved under passage by restriction  to the subset of diagonal blocks of a given size $[i:j]$.    In particular
restriction is a homomorphism  of Lie groups.  Beyond that we see that the preservation of matrix sums, products, and scalar products leads to preservation of polynomial 
functions, and thus to preservation for converging matrix power series.    Thus by continuity a power series defining  an analytic function $F$ on $\T_n$ or $\t_n$ applied to a diagonal block $A[i:j]$  yields the same result as the restriction $F(A)[i:j]$.  For examples: 
\begin{eqnarray}
   A^{-1}[i:j] &=& (A[i:j]) ^{-1},\qquad A\in \T_n;
  \\ 
  \exp(A)[i:j] &=&  \exp(A[i:j]),\qquad A\in \t_n;
  \\
  \log(A)[i:j] &=&  \log(A[i:j]),\qquad A\in \T_n; 
  \\
  (AB)[i:j] &=&  (A[i:j])(B[i:j]).\qquad A,B\in \t_n.
\end{eqnarray}

The log function will be important in our coming considerations, so (3) of the following remark will be useful.
\begin{remark}\label{R:blocks}
$(1)$ For the smallest diagonal blocks $A[i:i]$ with entry $a_{ii}$, it follows that $F(A)[i:i]= [F(a_{ii})]$, equivalently $F(A)_{ii}=F(A_{ii})$.  \\
$(2)$ For the matrix exponential function $exp:\t_n\to \T_n$, we have $(\exp(A))_{ii}=\exp(A_{ii})$.  \\
$(3)$ For $\log:\T_n\to \t_n$, we have $\log(A)_{ii}=\log(A_{ii})$ for all $A\in \T_n$. 
\end{remark}
\begin{proof}
Items (1) and (2) are straightforward consequences of the material preceding this remark.  For (3), we
note from (2) that $\exp((\log(A))_{ii})=(\exp(\log(A))_{ii})=A_{ii}$.  Applying $\log$ to both sides yields
$\log(A)_{ii}=\log(A_{ii})$.

\end{proof}

\subsection{Basic Results Regarding the Karcher Mean}
 In this section we gather together some basic properties of the Karcher mean as
 defined in equation (\ref{Karcher equation}) that will be important in what follows.
\begin{lemma}\label{thm: mean property}
Let \(\w=(w_1,\dots,w_m)\) be a set of weights and \((A_1,\dots,A_m)\in (\T_n)^m\). If \(X\)
is a $\w$-weighted Karcher mean, 
then the following properties hold.
\begin{enumerate}
	\item 
	$$\diag(X)=\diag (A_1^{w_1}\cdots A_m^{w_m}).$$
	Equivalently, the diagonal \((\overline{a}_{11},\dots,\overline{a}_{nn})\) of
	\(X\) is given by
	\[
	\overline{a}_{ii}=(a_{ii}^1)^{w_1}(a_{ii}^2)^{w_2}\cdots (a_{ii}^m)^{w_m},
	\]
	where \(a_{ij}^k\) is the \((i,j)\)-entry of \(A_k\). 
	\item For $1\le i\le j\le n$, the principal submatrix  $X[i:j]$ is a $\w$-weighted Karcher mean of 
	 $(A_1[i:j],\ldots, A_m[i:j]) $. 
	In particular, the $(i,j)$-entry of $X$ depends only on $\w$ and the entries in the principal submatrices $A_1[i:j],\ldots, A_m[i:j]$.  
\end{enumerate}

\end{lemma}
 
\begin{proof}
By assumption, $X$ satisfies the Karcher equation  
\[\sum_{i=1}^m w_i \log(X^{-1}A_i)=0.
\]
\begin{enumerate}
\item Taking  diagonal parts, we obtain
\[
\begin{aligned}
0 &=\diag\left(\sum_{i=1}^m w_i \log(X^{-1}A_i)\right)
=\sum_{i=1}^m w_i \diag(\log(X^{-1}A_i)) \\
&= \sum_{i=1}^m w_i \log(\diag(X^{-1}A_i)) 
=\sum_{i=1}^m w_i \log(\diag(X)^{-1}\diag(A_i)) \\
&=\sum_{i-1}^m -w_i\log(\diag X)+\sum_{i=1}^m w_i\log(\diag(A_i))=\log(\diag (A_1^{w_1}\cdots A_m^{w_m}))-\log(\diag(X)).
\end{aligned}
\]
The first equality of the second line follows from Remark \ref{R:blocks}(3).  Hence the first claim is proved. 

\item 
Taking the $[i:j]$-principal blocks
on both sides of  the Karcher equation and referencing the second paragraph of 2.1 gives
\[
	0 = \left(\sum_{i=1}^m w_i \log(X^{-1}A_i)\right)[i:j]
	=\sum_{i=1}^m w_i \log(X[i:j]^{-1}\, A_i[i:j] ). 
\]
Therefore, $X[i:j]$ is a $\w$-weighted Karcher mean of 
$(A_1[i:j],\ldots, A_m[i:j]) $.  \qedhere 
\end{enumerate}
\end{proof}

The following result shows that 
the Karcher means are bi-affine invariant. 
 
\begin{lemma} \label{thm: Karcher bi-affine}
Let \(\w=(w_1,\dots,w_m)\) be a set of weights. 
Let $X$ be a $\w$-weighted Karcher mean of \((A_1,\dots,A_m)\in (\T_n)^m\).
Then for any $P,Q\in\T_n$, $PXQ$ is 
a  $\w$-weighted Karcher mean of \((PA_1Q,\dots, PA_mQ )\).
\end{lemma}

\begin{proof}

By assumption, $X$
	satisfies the  Karcher equation \eqref{Karcher equation}. 
	Then 
	\[
	\begin{aligned}
		\sum_{i=1}^{m} w_i \log \left( (PXQ)^{-1} (PA_iQ) \right) 
		&=
		\sum_{i=1}^{m} w_i \log \left( Q^{-1}X^{-1} A_iQ \right)
		\\
		&= Q^{-1}\left(\sum_{i=1}^{m} w_i \log \left( X^{-1} A_i \right)\right) Q =0.
	\end{aligned}
	\]
Therefore, $PXQ$ is 
a  $\w$-weighted Karcher mean of \((PA_1Q,\dots, PA_mQ )\). 
\end{proof}

When $m=2$, the weighted Karcher mean exists uniquely and coincides with the weighted geometric mean on $\T_n$. 

\begin{theorem}
Suppose \(m=2\) and \(\w=(1-t,t)\) for some \(t\in(0,1)\). For \(A_1,A_2\in\T_n\), the \(\w\)-weighted Karcher mean exists uniquely and depends analytically on the entries of \(A_1\) and \(A_2\). More precisely,
\begin{align}
G(\w;A_1,A_2)
&= A_1^{1/2}\bigl(A_1^{-1/2}A_2A_1^{-1/2}\bigr)^tA_1^{1/2} \\
&= A_1(A_1^{-1}A_2)^t \\
&= (A_2A_1^{-1})^tA_1.
\end{align}
We denote
\[
A_1\#_tA_2 := G(\w;A_1,A_2).
\]
Moreover,
\[
A_1\#_t A_2 = A_2\#_{1-t}A_1.
\]
\end{theorem}

\begin{proof}
For \(A\in\T_n\), consider the Karcher equation
\[
(1-t)\log(X^{-1}I_n)+t\log(X^{-1}A)=0.
\]
Since \(\log(X^{-1}I_n)=-\log X\), this equation is equivalent to
\[
\log (X^{(1-t)/t}) = \log(X^{-1}A),
\]
whose unique solution in \(\T_n\) is \(X=A^t\). Since the matrix power map is analytic on \(\T_n\), the solution depends analytically on the entries of \(A\).

Applying the congruence invariance of the Karcher mean yields
\begin{align*}
G(\w;A_1,A_2)
&= A_1^{1/2}
   G\bigl(\w;I_n,A_1^{-1/2}A_2A_1^{-1/2}\bigr)
   A_1^{1/2} \\
&= A_1^{1/2}
   \bigl(A_1^{-1/2}A_2A_1^{-1/2}\bigr)^t
   A_1^{1/2}.
\end{align*}
Using the similarity property, we obtain
\[
G(\w;A_1,A_2)= A_1 [A_1^{-1/2}
\bigl(A_1^{-1/2}A_2A_1^{-1/2}\bigr)^t
A_1^{1/2}]
=
A_1(A_1^{-1}A_2)^t.
\]
Similarly,
\[
G(\w;A_1,A_2)
=
(A_2A_1^{-1})^tA_1.
\]
Hence the weighted Karcher mean is uniquely determined and analytic in the entries of \(A_1\) and \(A_2\).

Finally,
\begin{align*}
A_2\#_{1-t}A_1
&= A_2(A_2^{-1}A_1)^{1-t} \\
&= A_2(A_2^{-1}A_1)(A_1^{-1}A_2)^t \\
&= A_1(A_1^{-1}A_2)^t \\
&= A_1\#_tA_2. \qedhere
\end{align*}
\end{proof}

\section{Parlett Recurrence, Divided-Difference Functions, and Analytic Functions on Triangular Matrices}
In this section we quickly review some standard methods of extending functions, particularly analytic functions, defined on a domain of real (or complex) numbers,
to matrix functions on $\T_n$.  Our primary concern is involves the computation of the logarithmic function on $\T_n$.  
For further reference, see \cite{dB05}.

\subsection{Parlett Recurrence}

In \cite{Par76} B.~N.~Parlett proposed a recursive scheme for extending an analytic function on a complex domain to an analytic function on upper triangular matrices over $\mathbb{C}$ whose diagonal entries lie in the domain.

From a
map \(t\mapsto f(t)\) we seek a function \(F=f(T)\) defined from upper triangular matrices with
distinct diagonal entries to \(\T_n\). Parlett uses the commutativity relation \(FT=TF\) to
derive the following recursion equation.

\medskip

\noindent \textbf{Algorithm 3.1. (Parlett recurrence for upper triangular matrices)}
\begin{equation}\label{Parlett recurrence}
	f_{ij}
	=
	t_{ij}\frac{f_{ii}-f_{jj}}{t_{ii}-t_{jj}}
	+
	\sum_{k=i+1}^{j-1}
	\frac{f_{ik}t_{kj}-t_{ik}f_{kj}}{t_{ii}-t_{jj}},
	\qquad i<j.
\end{equation}
In the preceding \(f_{ij}\) resp.\ \(t_{ij}\) represents the \(ij\)-entry of \(F=f(T)\) resp.\ \(T\). From Equation
\eqref{Parlett recurrence} we see that any element of \(F\) can be calculated so long as all the elements to the left
and below it are known. Thus the recurrence allows us to compute \(F\) a superdiagonal at a
time, starting with the diagonal elements \(f_{ii}=f(t_{ii})\).

The original Parlett recurrence  is 
formulated in block form. We record here the $2\times 2$ block version to be used in this paper.

\noindent \textbf{Algorithm 3.2 (Parlett recurrence for $2\times 2$ block upper triangular matrices).} 
Let $f$ be an analytic function, and let \(T=[T_{ij}]\) be a $2\times 2$ block upper triangular matrix. Then \(F=f(T)=[F_{ij}]\) is also block upper triangular, with diagonal blocks given by \(F_{ii}=f(T_{ii})\) for $i=1,2$, and the off-diagonal block \(F_{12}\) satisfies that
\begin{equation}\label{Parlett recurrence block}
	T_{11}F_{12}-F_{12}T_{22}
	=
	F_{11}T_{12}-T_{12}F_{22}.
\end{equation}

Our interest in both versions of the Parlett recurrence concerns the computation of the iteration arisen from the Karcher equation on the upper triangular matrices.

\subsection{Analytic functions on upper triangular matrices}

Let $f$ be an analytic function  on a complex domain $U$. 
Let $T=[t_{ij}]$ be a upper   triangular matrix whose diagonal entries are inside $U$.  The explicit expression of $f(T)$ can be obtained through the divided difference functions.  

Given $a_1, a_2, \ldots\in U$, we define $f^{[0]}(a_1) :=f(a_1)$ and for $k\ge 1$: 
\begin{equation}
	f^{[k]}(a_1,\ldots, a_{k+1}) :=
	\begin{cases}
		\dfrac{f^{[k-1]}(a_1,\ldots, a_k)-f^{[k-1]}(a_2,\ldots, a_{k+1})}{a_1-a_{k+1}}, &a_1\ne a_{k+1};
		\\[1em]
		\dfrac{\partial f^{[k-1]}}{\partial x_1}(a_1, a_2,\cdots,a_k), &a_1=a_{k+1}.
	\end{cases}
\end{equation}
We call $f^{[k]}$ \emph{the $k$-th divided difference function of $f$.}

The divided difference functions are widely used in recursive numerical computations of the coefficients for Newton interpolation polynomials. 
The following known formulas can be proved by induction. See \cite{dB05} for a more detailed treatment. 
\begin{enumerate}
    \item
    If $\mathcal{C}$ is a contour on $U$ that embrace $a_1,\ldots, a_{k+1}$,
then 
\begin{equation}
	f^{[k]}(a_1,\ldots, a_{k+1})
\, =\, 
\dfrac{1}{2\pi\i} \int_{\mathcal{C}} \dfrac{f(z)}{(z-a_1)\cdots (z-a_{k+1})}\, dz. 
\end{equation}

\item 
If $a_1,\ldots, a_{k+1}$ are distinct, then
\begin{equation}
	f^{[k]}(a_1,\ldots, a_{k+1})
	\, =\, \sum_{i=1}^{k+1} \dfrac{f(a_i)}{\underset{j\ne i}{\prod_{j=1}^{k+1}}  (a_i-a_j)}.
\end{equation}

\end{enumerate}

Now we use the divided difference function of $f$ to describe $f(T)$ when all of the diagonal entries $t_{ii}$ are inside $U$. 

A  sequence of integers \(\gamma=(i_1,\ldots, i_{k+1})\) is called an \((i,j)\)-path if
\[
i=i_1<i_2<\cdots<i_{k+1}=j.
\]
We refer to \(i_1,\ldots,i_{k+1}\) as the nodes of \(\gamma\), to
$
(i_1,i_2),\ldots,(i_k,i_{k+1})
$
as the arcs of \(\gamma\), and to \(k\) as the length of \(\gamma\), denoted by \(|\gamma|=k\).

Let \(\Gamma_{ij}\) denote the set of all \((i,j)\)-paths. Then
$
|\Gamma_{ij}|=2^{j-i-1}.
$

Given two indices \(1\le i<j\le n\), and an \((i,j)\)-path
\(\gamma=(i_1,\ldots,i_k,i_{k+1})\), we denote
\begin{eqnarray}
	f^{[\gamma]}(T)
	&:=& 
	f^{[|\gamma|]}\bigl(t_{i_1i_1},\, \ldots,\, t_{i_ki_k},\, t_{i_{k+1}i_{k+1}}\bigr),
\label {f gamma}\\
	\Pi_{\gamma}(T)
	&:=&
	t_{i_1i_2}\, t_{i_2i_3}\, \cdots \, t_{i_ki_{k+1}}. 
\label{Pi gamma}
\end{eqnarray}

\begin{theorem}\label{thm:Tn-analytic}
	Let \(T=[t_{ij}]\) be an $n\times n$ upper triangular matrix. Then \(f (T)=[f_{ij}]\) is also upper triangular, where 
	\[
	f_{ii}=f (t_{ii}), \qquad 1\le i\le n,
	\]
	and, for \(1\le i<j\le n\),
	\begin{align}
		f_{ij}
		&=
		\sum_{\gamma\in\Gamma_{ij}} 
		f^{[\gamma]}(T)\,\Pi_{\gamma}(T) 
		\label{fij_expression} \\
		&=
		\sum_{k=1}^{j-i}
		\ \sum_{i=i_1<\cdots<i_{k+1}=j}
		f^{[k]}\bigl(t_{i_1i_1},\, \ldots,\, t_{i_{k+1}i_{k+1}}\bigr)\,
		t_{i_1i_2}\, t_{i_2i_3}\, \cdots\,  t_{i_ki_{k+1}}. \notag
	\end{align}
\end{theorem}

\begin{proof} The formula of $f_{ii}$ is straightforward. 
It suffices to prove the formula \eqref{fij_expression} for upper triangular \(T\) with distinct diagonal entries and extend the formula continuously to non-distinct cases. 
We apply induction on \(n\) using Parlett recurrence. The cases \(n=1\) and \(n=2\) are immediate. Assume \(n\ge 3\), and suppose that \eqref{fij_expression} holds on \(\T_p\) for all \(p<n\).
	
	By Section 2.1,  
	\[
	f(T)[1:n-1]=f(T[1:n-1]), \qquad f(T)[2:n]=f(T[2:n]),
	\]
	and hence \eqref{fij_expression} already holds for all entries except possibly the \((1,n)\)-entry.
	
	Applying the Parlett recurrence formula \eqref{Parlett recurrence} to the \((1,n)\)-entry, we obtain
	\begin{align*}
		f_{1n}
		&=
		\frac{f_{11}-f_{nn}}{t_{11}-t_{nn}}\,t_{1n}
		+\sum_{\ell=2}^{n-1}\frac{f_{1\ell}t_{\ell n}-t_{1\ell}f_{\ell n}}{t_{11}-t_{nn}} \\
		&=
		f^{[1]} (t_{11},t_{nn})\,t_{1n}
		+\dfrac{\sum_{\ell=2}^{n-1} f_{1\ell}t_{\ell n}}{t_{11}-t_{nn}}
		-\dfrac{\sum_{\ell=2}^{n-1} t_{1\ell}f_{\ell n}}{t_{11}-t_{nn}}.
	\end{align*}
	Substituting the inductive formulas for \(f_{1\ell}\) and \(f_{\ell n}\) yields
	\begin{align*}
		f_{1n}
		&=
		f^{[1]} (t_{11},t_{nn})\,t_{1n} 
		+\dfrac{\displaystyle\sum_{\ell=2}^{n-1} \sum_{\gamma\in\Gamma_{1\ell}} f^{[\gamma]}(T) \,  \Pi_{\gamma}(T) \, t_{\ell n}}{t_{11}-t_{nn}}
		-\dfrac{\displaystyle \sum_{\ell=2}^{n-1} \sum_{\gamma\in\Gamma_{\ell n}} t_{1\ell} \, f^{[\gamma]}(T) \,  \Pi_{\gamma}(T)
		}{t_{11}-t_{nn}}.
	\end{align*}
	Reindexing the sums in numerators and using the recursive identity
	\[
	f^{[k+1]}(a_1,\ldots,a_{k+2})
	=
	\frac{f^{[k]}(a_1,\ldots,a_{k+1})-f^{[k]}(a_2,\ldots,a_{k+2})}{a_1-a_{k+2}},
	\]
	we obtain
	\[
	f_{1n}
	=
	\sum_{\gamma\in\Gamma_{1n}} f^{[\gamma]}(T) \,  \Pi_{\gamma}(T).
	\]
	Thus \eqref{fij_expression} holds for the \((1,n)\)-entry, and therefore it holds for all entries of \(f(T)\) when \(T\in\T_n\). This completes the proof of the theorem.
\end{proof}

\begin{remark}
Set $f$ to be the power function $\T_n\to\T_n$ such that $X\mapsto X^t$, the log function $\log: \t_n\to \T_n$, or the exp function $\exp: \T_n\to \t_n$, we  can obtain the explicit expression of the corresponding function for appropriate triangular matrices. 
\end{remark}

\begin{remark}
When the spectrum of a square matrix $A$ is contained in an open domain of an analytic function $f$, we may use Schur triangulation to get $A=UTU^*$ for a unitary $U$ and a triangular $T$ and get $f(A)= Uf(T)U^*$, in which $f(T)$ can be described by Theorem \ref{thm:Tn-analytic}. 
\end{remark}

As an application, we derive the inverse formula of a unipotent upper triangular matrix.

\begin{lemma}\label{thm:X-inv-A}
	Let \(X=\mtx{x_{ij}}\in\T_n\) be unipotent upper triangular. Then \(X^{-1}\in\T_n\) is also unipotent upper triangular, and
	\begin{equation}
		(X^{-1})_{ij}
		=
		\sum_{\gamma\in\Gamma_{ij}}
		(-1)^{|\gamma|}\Pi_{\gamma}(X),
		\qquad 1\le i<j\le n.
		\label{Xinv term}
	\end{equation}
	If, in addition, \(A=\mtx{a_{ij}}\in\T_n\), then \(X^{-1}A\in\T_n\),
	\[
	\diag(X^{-1}A)=\diag(A),
	\]
	and, for \(1\le i<j\le n\),
	\begin{equation}
		(X^{-1}A)_{ij}
		=
		a_{ij}
		+
		\sum_{\ell=i+1}^{j}
		\sum_{\gamma\in\Gamma_{i\ell}}
		(-1)^{|\gamma|}\Pi_{\gamma}(X)\,a_{\ell j}.
		\label{XinvA term}
	\end{equation}
\end{lemma}

\begin{proof}
Let $f(x)=x^{-1}$ for $x\ne 0$. We prove that for positive numbers  $a_1, \ldots, a_{k+1}$, 
\begin{equation}\label{x_inv_tri}
    f^{[k]}(a_1, \ldots, a_{k+1})
=(-1)^{k} \dfrac{1}{a_1\cdots a_{k+1}}.
\end{equation}
Assume first that $a_1, \ldots, a_{k+1}$ are distinct. The formula is clearly true for $k=0$. Suppose the 
formula is true for $k-1$. Then
\begin{align*}
    f^{[k]}(a_1, \ldots, a_{k+1})
&= \dfrac{f^{[k-1]}(a_1, \ldots, a_{k})-f^{[k-1]}(a_2, \ldots, a_{k+1})}{a_1-a_{k+1}}
\\
&= (-1)^{k-1} \dfrac{\frac{1}{a_1\ldots a_{k}}-\frac{1}{a_2\ldots a_{k+1}}}{a_1-a_{k+1}}
=(-1)^{k} \dfrac{1}{a_1\cdots a_{k+1}}.
\end{align*}
Therefore, by induction, \eqref{x_inv_tri} is true for distinct  
positive  $a_1, \ldots, a_{k+1}$ and for all $k$, and it is also true for non-distinct cases by continuity. 

Now by \eqref{fij_expression}, when $X$ is unipotent upper triangular, we have $x_{11}=\cdots=x_{nn}=1$, so that for $1\le i< j\le n$, 
\[
(X^{-1})_{ij}
= \sum_{\gamma\in\Gamma_{ij}} 
		f^{[\gamma]}(T)\,\Pi_{\gamma}(T)
=  \sum_{\gamma\in\Gamma_{ij}} 
		(-1)^{k} \,\Pi_{\gamma}(T).
\]
So \eqref{Xinv term} is proved.
	
Finally, let \(A\in\T_n\). Since \(X\) is unipotent,
	\[
	\diag(X^{-1}A)=\diag(X^{-1})\diag(A)=\diag(A).
	\]
	Moreover, for \(1\le i<j\le n\),
	\[
	(X^{-1}A)_{ij}
	=
	\sum_{\ell=i}^{j}(X^{-1})_{i\ell}a_{\ell j}
	=
	a_{ij}
	+\sum_{\ell=i+1}^{j}(X^{-1})_{i\ell}a_{\ell j}.
	\]
	Substituting \eqref{Xinv term} into the last expression yields \eqref{XinvA term}.
\end{proof}

Lemma \ref {thm:X-inv-A} will be used later to find the expression of the Karcher mean on $\T_n$ using recursive algorithms. It can also be used to find the explicit analytic expression of the geometric mean  on $\T_n$ through the formula 
$$A_1\#_t A_2 = A_1(A_1^{-1}A_2)^t.$$

We may apply Parlett recurrence and the divided difference function to obtain the perturbation behavior  at the $(1,n)$ entry for any analytic function on a set of triangular matrices.

\begin{lemma}\label{thm: log 1n formula}
Let $T=\mtx{t_{ij}}$ be an upper triangular matrix whose diagonal entries lie in the interior of the domain of an analytic function $f$. Then, for every $x\in\C$,
\begin{equation}\label{log formula}
    f(T+xE_{1n})
    =
    f(T)
    +
    x\, f^{[1]}(t_{11},t_{nn})E_{1n}.
\end{equation}
\end{lemma}

\begin{proof}
Set
\[
F=f(T)=\mtx{f_{ij}},
\qquad
F^\ast=f(T+xE_{1n})=\mtx{f_{ij}^\ast}.
\]
By the discussion in Section~2.1, the matrices $F$ and $F^\ast$ differ only in their $(1,n)$-entries. Hence it suffices to determine $f_{1n}^\ast-f_{1n}$.

Assume first that $t_{11}\neq t_{nn}$. Applying the Parlett recurrence formula \eqref{Parlett recurrence} to the $(1,n)$-entry yields
\[
f_{1n}
=
t_{1n}\frac{f_{11}-f_{nn}}{t_{11}-t_{nn}}
+
\sum_{k=2}^{n-1}
\frac{f_{1k}t_{kn}-t_{1k}f_{kn}}{t_{11}-t_{nn}}.
\]
Since the perturbation $xE_{1n}$ modifies only the $(1,n)$-entry of $T$, we similarly obtain
\[
\begin{aligned}
f_{1n}^\ast
&=
(t_{1n}+x)\frac{f_{11}-f_{nn}}{t_{11}-t_{nn}}
+
\sum_{k=2}^{n-1}
\frac{f_{1k}t_{kn}-t_{1k}f_{kn}}{t_{11}-t_{nn}} \\
&=
f_{1n}
+
x\,\frac{f_{11}-f_{nn}}{t_{11}-t_{nn}} \\
&=
f_{1n}
+
x\,f^{[1]}(t_{11},t_{nn}).
\end{aligned}
\]
Therefore,
\[
F^\ast-F
=
x\,f^{[1]}(t_{11},t_{nn})E_{1n},
\]
which proves \eqref{log formula} when $t_{11}\neq t_{nn}$.

The remaining case $t_{11}=t_{nn}$ follows by continuity of the divided difference $f^{[1]}(x_1, x_2)$.
\end{proof}

We will need the above perturbation formula
for
the log function on $\T_n$.

\section{The Karcher Mean on \(2\times 2\) Upper Triangular Matrices}

In this section we investigate Parlett recurrence and the Karcher mean for the group
\(\T_2\), which consists of all \(2\times 2\) upper triangular matrices with real entries and with
positive entries on the diagonal. 
The explicit Karcher mean formula on $\T_2$ has been given in \cite{Hu26}. 

We first identify the matrix logarithmic function on \(\T_2\) as that induced via Parlett
recurrence by \(x\mapsto \log  x\) on the real numbers. Let
\[
T=
\begin{bmatrix}
t_{11} & t_{12}\\
0 & t_{22}
\end{bmatrix},
\qquad \text{where } t_{11},t_{22}>0.
\]
Then by Lemma \ref{thm: log 1n formula},
\begin{equation}\label{log 2x2}
\log(T)=
\begin{bmatrix}
\log (t_{11}) &
\log^{[1]} (t_{11}, t_{nn}) \, t_{12}\\[1em]
0 & \log (t_{22})
\end{bmatrix},
\end{equation}
where $\log^{[k]} (x_1, x_2)$ is the $k$-th divided difference of log function. 
We have the following
proposition.

\begin{proposition}
The function \(\log : \T_2\to \t_2\) is given by equations \eqref{log 2x2} 
and is an analytic diffeomorphism.
\end{proposition}

We turn now to the Karcher mean. Let \(\{w_1,\dots,w_n\}\) be weights, that is to say
\(0<w_i<1\) for each \(w_i\) and \(\sum_{i=1}^n w_i=1\). Let \(\{A_1,\dots,A_n\}\) denote \(n\) elements of \(\T_2\),
where
\[
A_k=
\begin{bmatrix}
a_{11}^k & a_{12}^k\\[1ex]
0 & a_{22}^k
\end{bmatrix}
\qquad \text{for } 1\le k\le n.
\]
By Lemma \ref{thm: mean property} if a Karcher mean \(X\) exists, then
\[
x_{11}=\overline{a}_{11}=(a_{11}^1)^{w_1}\cdots (a_{11}^n)^{w_n}
\quad \text{and} \quad
x_{22}=\overline{a}_{22}=(a_{22}^1)^{w_1}\cdots (a_{22}^n)^{w_n}.
\]
Hence to find a Karcher mean we seek an \(X\) of the form
\[
X=
\begin{bmatrix}
\overline{a}_{11} & x_{12}\\
0 & \overline{a}_{22}
\end{bmatrix},
\]
and thus
\begin{equation}\label{X_inv}
X^{-1}=
\begin{bmatrix}
\dfrac{1}{\overline{a}_{11}} & -\dfrac{x_{12}}{\overline{a}_{11}\overline{a}_{22}}\\[1em]
0 & \dfrac{1}{\overline{a}_{22}}
\end{bmatrix}
\end{equation}
and
\[
X^{-1}A_k=
\begin{bmatrix}
\dfrac{a_{11}^k}{\overline{a}_{11}} &
\dfrac{a_{12}^k}{\overline{a}_{11}}-\dfrac{x_{12}a_{22}^k}{\overline{a}_{11}\overline{a}_{22}}\\[1em]
0 & \dfrac{a_{22}^k}{\overline{a}_{22}}
\end{bmatrix}.
\]
From these results we obtain
\begin{equation}\label{log(XinvAk)}
\log(X^{-1}A_k)=
\begin{bmatrix}
\log  \!\left(\dfrac{a_{11}^k}{\overline{a}_{11}}\right) &
\alpha_k\!\left(\dfrac{a_{12}^k}{\overline{a}_{11}}-\dfrac{x_{12}a_{22}^k}{\overline{a}_{11}\overline{a}_{22}}\right)\\[2ex]
0 & \log  \!\left(\dfrac{a_{22}^k}{\overline{a}_{22}}\right)
\end{bmatrix},
\qquad
\alpha_k=
\log^{[1]} \left(
\dfrac{a_{22}^k}{\overline{a}_{22}}, \dfrac{a_{11}^k}{\overline{a}_{11}}\right).
\end{equation}

We next consider the sum \(\sum_{k=1}^m w_k \log(X^{-1}A_k)\) and recall that \(X\) is a Karcher mean
if and only if this sum is the \(0\)-matrix. We note from equation \eqref{log(XinvAk)} the sum is always
\(0\) in the \((1,1)\) and \((2,2)\) entries, so the only thing that needs to be checked is the \((1,2)\)
entry. We solve the appropriate equation for this to be the case in the \((1,2)\)-coordinate:
\[
0=\sum_{k=1}^m w_k \alpha_k
\left(
\frac{a_{12}^k}{\overline{a}_{11}}-\frac{x_{12}a_{22}^k}{\overline{a}_{11}\overline{a}_{22}}
\right)
=
\frac{1}{\overline{a}_{11}}\sum_{k=1}^m w_k\alpha_k a_{12}^k
-
\frac{x_{12}}{\overline{a}_{11}\overline{a}_{22}}\sum_{k=1}^m w_k\alpha_k a_{22}^k.
\]

\begin{theorem}\label{thm: Karcher mean UT2}
Let \(\w=(w_1,\dots,w_m)\) be a set of weights and \((A_1,\dots,A_m)\in (\T_2)^m\). Then there exists a unique Karcher mean given by
\[
X=
\begin{bmatrix}
\overline{a}_{11} & x_{12}\\
0 & \overline{a}_{22}
\end{bmatrix}, \qquad
\text{where } 
x_{12}
=
\overline{a}_{22}
\left(
\frac{\sum_{k=1}^m w_k \alpha_k a_{12}^k}
{\sum_{k=1}^m w_k \alpha_k a_{22}^k}
\right),
\]
\(\overline{a}_{11}\) and \(\overline{a}_{22}\) are as given in Lemma \ref{thm: mean property}, and $\alpha_k$ is given in \eqref{log(XinvAk)}. 
\end{theorem}

\begin{proof}
The correctness of the entries on the diagonal follow from Lemma \ref{thm: mean property} and the last
assertion about \(x_{12}\) follows from the straightforward solution for \(x_{12}\) in the equation just
preceding this theorem.
\end{proof}

\section{Karcher Mean on $\T_n$}

\subsection{Existence, Uniqueness, and Analyticity}

We have shown that 
the Karcher mean on $\T_2$ uniquely exists and is analytic. 
Applying Parlett recurrence, we are able to extend this result to the general theory on $\T_n$.

\begin{theorem}\label{thm: mean existence uniqueness}
	(Existence, uniqueness, and analyticity of the Karcher mean on $\T_n$) 
Let $\w =(w_1,\ldots,w_m)\in (0,1)^m$ be weights such that $w_1+\cdots+w_m=1$.
Let $(A_1,\ldots, A_m)\in (\T_n)^m.$
	Then there exists a unique $\w$-weighted Karcher mean $X\in\T_n$ that satisfies the Karcher equation \eqref{Karcher equation}. 
	Moreover, $X$ depends analytically on the entries of $A_1,\ldots, A_m$.
\end{theorem}

\begin{proof}
	Set
	\begin{align}
		D &:= \diag(A_1^{w_1}\cdots A_m^{w_m})=\diag(d_1,\ldots,d_n), \label{D}\\
		A_{(i)} &:= D^{-1}A_i=\bigl[a^{(i)}_{pq}\bigr], \qquad i=1,\ldots,m, \label{D_inv_A}
	\end{align}
	and, for \(1\le p<q\le n\),
	\begin{equation} \label{c_pq}
		c_{pq}:=\sum_{i=1}^m w_i\, a^{(i)}_{qq}\, \log^{[1]}\!\bigl(a^{(i)}_{pp},a^{(i)}_{qq}\bigr).  
	\end{equation}
	Since each \(A_i\in \T_n\), both \(D\) and \(A_{(i)}\) depend analytically on the entries of \(A_1,\ldots,A_m\). Moreover, by Lemma~2.2, if
	\[
	X'=G(\w;A_{(1)},\ldots,A_{(m)})
	\]
	exists, then
	\[
	G(\w;A_1,\ldots,A_m)=DX'.
	\]
	Therefore, it suffices to prove the existence, uniqueness, and analyticity of the \(\w\)-weighted Karcher mean of \((A_{(1)},\ldots,A_{(m)})\).
	
	We proceed by induction on \(n\). The case \(n=1\) is immediate.
	
	Assume now that \(n>1\) and that the statement has already been established for \(n-1\). By the induction hypothesis, there exists a unique matrix \(X_0'\in \T_n\) such that
	\begin{align}
		X_0'[1:n-1] &= G\bigl(\w;A_{(1)}[1:n-1],\ldots,A_{(m)}[1:n-1]\bigr), \label{eq:proof-block1}\\
		X_0'[2:n]   &= G\bigl(\w;A_{(1)}[2:n],\ldots,A_{(m)}[2:n]\bigr), \label{eq:proof-block2}
	\end{align}
	the \((1,n)\)-entry of \(X_0'\) is \(0\), and \(X_0'\) depends analytically on the entries of \(A_{(1)},\ldots,A_{(m)}\). The assumptions of $X_0'$ imply that 
	\[\diag(X_0')= \diag(A_{(1)}^{w_1}\cdots A_{(m)}^{w_m})
	=\diag(D^{-1})\diag(A_1^{w_1}\cdots A_m^{w_m})=I_n. 
	\]	
	
	Now define
	\begin{equation}
		Y:=\sum_{i=1}^m w_i \log\!\bigl((X_0')^{-1}A_{(i)}\bigr). \label{eq:proof-Y}
	\end{equation}
	By \eqref{eq:proof-block1} and \eqref{eq:proof-block2}, the Karcher equation holds on the two principal \((n-1)\times (n-1)\) diagonal blocks, and hence
	\[
	Y[1:n-1]=0
	\qquad\text{and}\qquad
	Y[2:n]=0.
	\]
	Since \(Y\) is upper triangular, these two identities force all entries of \(Y\) to vanish except possibly the \((1,n)\)-entry. Thus
	\begin{equation}
		Y=y_{1n}E_{1n} \label{eq:proof-Y-shape}
	\end{equation}
	for some \(y_{1n}\in\R\), and \(y_{1n}\) depends analytically on the entries of \(A_{(1)},\ldots,A_{(m)}\).
	
	We next determine the \((1,n)\)-entry of the desired Karcher mean. By Lemma~\ref{thm: mean property}(2), any Karcher mean of \((A_{(1)},\ldots,A_{(m)})\), if it exists, must agree with \(X_0'\) on all entries except possibly the \((1,n)\)-entry. We may therefore write
	\begin{equation}
		X' = X_0' + xE_{1n} \label{eq:proof-Xprime}
	\end{equation}
	for some \(x\in\R\), and seek \(x\) so that \(X'\) satisfies the Karcher equation
	\begin{equation}
		\sum_{i=1}^m w_i \log\!\bigl((X')^{-1}A_{(i)}\bigr)=0. \label{eq:proof-Karcher-eq}
	\end{equation}
	
	Since \(X_0'\in \T_n\) and
	$\diag(X_0')=I_n, $
	we have
	\[
	(X')^{-1}=(X_0'+xE_{1n})^{-1}=(X_0')^{-1}-xE_{1n}.
	\]
	Hence
	\[
	(X')^{-1}A_{(i)}
	=\bigl((X_0')^{-1}-xE_{1n}\bigr)A_{(i)}
	=(X_0')^{-1}A_{(i)}-xE_{1n}A_{(i)}.
	\]
	Because \(A_{(i)}\) is upper triangular, \(E_{1n}A_{(i)}=a^{(i)}_{nn}E_{1n}\), and therefore
	\begin{equation}
		(X')^{-1}A_{(i)}=(X_0')^{-1}A_{(i)}-x a^{(i)}_{nn}E_{1n}. \label{eq:proof-perturb}
	\end{equation}
	Also,
	\[
	\diag\bigl((X_0')^{-1}A_{(i)}\bigr)=\diag(A_{(i)})=\diag\bigl(a^{(i)}_{11},\ldots,a^{(i)}_{nn}\bigr).
	\]
	Applying Lemma \ref{thm:X-inv-A} for an \(E_{1n}\)-perturbation of log function, we obtain
	\[
	\log\!\bigl((X')^{-1}A_{(i)}\bigr)
	=
	\log\!\bigl((X_0')^{-1}A_{(i)}\bigr)
	-
	x a^{(i)}_{nn} \log^{[1]} \!\bigl(a^{(i)}_{11},a^{(i)}_{nn}\bigr)E_{1n}.
	\]
	Summing over \(i\) with weights \(w_i\) and using \eqref{eq:proof-Y} and \eqref{eq:proof-Y-shape}, we get
	\begin{align*}
		0
		&=\sum_{i=1}^m w_i \log\!\bigl((X')^{-1}A_{(i)}\bigr) \\
		&=\sum_{i=1}^m w_i \log\!\bigl((X_0')^{-1}A_{(i)}\bigr)
		-
		x\sum_{i=1}^m w_i a^{(i)}_{nn} \log^{[1]} \!\bigl(a^{(i)}_{11},a^{(i)}_{nn}\bigr)E_{1n} \\
		&=Y-xc_{1n}E_{1n} \\
		&=(y_{1n}-xc_{1n})E_{1n}.
	\end{align*}
	Thus
	\begin{equation}
		x=\frac{y_{1n}}{c_{1n}}. \label{eq:proof-x}
	\end{equation}
	
	Since each \(a^{(i)}_{11}>0\) and \(a^{(i)}_{nn}>0\), and since \(\log^{[1]} (s,t)>0\) for \(s,t>0\), we have \(c_{1n}>0\). Hence \eqref{eq:proof-x} determines a unique real number \(x\). Because both \(y_{1n}\) and \(c_{1n}\) depend analytically on the entries of \(A_{(1)},\ldots,A_{(m)}\), so does \(x\). Consequently,
	\[
	X'=X_0'+xE_{1n}
	\]
	is uniquely determined, belongs to \(\T_n\), satisfies the Karcher equation \eqref{eq:proof-Karcher-eq}, and depends analytically on the entries of \(A_{(1)},\ldots,A_{(m)}\).
	
	We have therefore proved the existence, uniqueness, and analyticity of the \(\w\)-weighted Karcher mean of \((A_{(1)},\ldots,A_{(m)})\). Multiplying by \(D\), we conclude that
	\[
	G(\w;A_1,\ldots,A_m)=DX'
	\]
	exists uniquely and depends analytically on the entries of \(A_1,\ldots,A_m\). This completes the induction and the proof.
\end{proof}

The constants $c_{pq}$ defined in \eqref{c_pq} are essential for the expressions of the Karcher means on $\T_n$. By \cite[Lemma 2.2]{Hu26}, we have
\[c_{pq}\ge 1,\qquad 
1\le p< q\le n.
\]

The preceding proof yields a finite recursive procedure for computing the Karcher mean on $\T_n$, assuming that the Karcher mean on $\T_{n-1}$
is known.  The explicit computation carried out by applying Parlett recurrence and divided difference functions will be shown next.

\subsection{An Iterative Algorithm for the Karcher Mean }

We are now in a position to present the analytic formula of the Karcher means on $\T_n$ in detail. 

\begin{theorem}\label{thm: Tn Karcher mean}
Let $\w =(w_1,\ldots,w_m)\in (0,1)^m$ be weights such that $w_1+\cdots+w_m=1$.
Let $(A_1,\ldots, A_m)\in (\T_n)^m.$ Then the $\w$-weighted Karcher mean
$X=\mtx{x_{ij}}\in\T_n$ is given by 
\[X=DX',\]
where $D=\diag(A_1^{w_1}\cdots A_m^{w_m})$, and
\[
X'=\mtx{x_{ij}'}
\]
is an unipotent upper triangular matrix whose entries may be computed recursively as follows: $\diag(X')=I_n$, and for each $1\le p<q\le n$, 
\begin{equation}
x_{pq}' = \frac{1}{c_{pq}}\sum_{i=1}^{m} w_i \, 
\Bigl ( 
\sum_{\gamma\in\Gamma_{pq} } 
\log^{[\gamma]} \bigl (A_{(i)} \bigr ) \;  \Pi_{\gamma} \bigl(  {X'_{pq}}^{-1} A_{(i)}    \bigr ) \Bigr ),
\label{x_pq'}
\end{equation}
where $X'_{pq}$ is the unipotent upper triangular matrix such that
\[
(X'_{pq})_{ij}=x_{ij}'\qquad \text{for } p\le i<j\le q,\ (i,j)\ne (p,q),
\]
and all other entries are $0$,
$\log^{[\gamma]}$ is defined in \eqref{f gamma} for $f(x)=\log(x)$, and $\Pi_{\gamma}$ is defined in \eqref{Pi gamma}.
The entries of ${X'_{pq}}^{-1} A_{(i)}$ are given in \eqref{XinvA term}. 
\end{theorem}

\begin{proof}
In Theorem \ref{thm: mean existence uniqueness},
we obtained  $X=DX'$ for the unipotent upper triangular matrix
\[
X'=\mtx{x_{ij}'} =G(\w;A_{(1)},\ldots,A_{(m)}),
\]
and Section 2.1 shows that $x_{pq}'$ for $1\le p<q\le n$ is the upper-right corner entry of
\[
G(\w;A_{(1)}[p:q],\ldots,A_{(m)}[p:q]).
\] 
By \eqref{eq:proof-Y} and \eqref{eq:proof-x}, 
\[
\begin{aligned}
x_{pq}' &=\frac{1}{c_{pq}} \Bigl(\sum_{i=1}^m w_i \log\!\bigl( {X_{pq}'}^{-1}A_{(i)} \bigr)\Bigr)_{pq}
 = \frac{1}{c_{pq}}
\sum_{i=1}^m w_i \Bigl(\log\!\bigl( {X_{pq}'}^{-1}A_{(i)}\bigr)\Bigr)_{pq}.
\end{aligned}
\]
Then \eqref{x_pq'} follows from Theorem \ref{thm:Tn-analytic}. 
\end{proof}

We now specialize the preceding results to the cases $n=3$ and $n=4$. 
 
\begin{ex}\label{thm: Log 3x3}
By Theorem \ref{thm:Tn-analytic},
the logarithm of a $3\times 3$ matrix $T=\mtx{t_{ij}} \in\T_3$ is 
$$
\log(T) =\mtx{f_{ij}} \in\t_3, 
$$
where
\[
\begin{aligned}
f_{ii} &= \log(t_{ii}),\qquad i=1,2,3,
\\
		f_{12} &=\log^{[1]} (t_{11},t_{22}) \, t_{12},
		\qquad 
		f_{23} = \log^{[1]} (t_{22},t_{33}) \, t_{23},
\\
		f_{13} &=
		\log^{[1]} (t_{11},t_{33})\,t_{13} + 
		\log^{[2]} (t_{11},t_{22},t_{33}) \, t_{12}t_{23}. 
\end{aligned}
\]
If $t_{11}, t_{22}, t_{33}$ are pairwise distinct, then 
\[
f_{12} = \dfrac{\log(t_{11})-\log(t_{22})}{t_{11}-t_{22}},
\qquad 
f_{23} = \dfrac{\log(t_{22})-\log(t_{33})}{t_{22}-t_{33}},
\]
and 
\[
\begin{aligned}
f_{13}
&=
\frac{\log(t_{11})-\log(t_{33})}{t_{11}-t_{33}}\,t_{13}
\\
&\qquad \qquad 
+
\left[
\frac{\log(t_{11})}{(t_{11}-t_{22})(t_{11}-t_{33})}
+\frac{\log(t_{22})}{(t_{22}-t_{11})(t_{22}-t_{33})}
+\frac{\log(t_{33})}{(t_{33}-t_{11})(t_{33}-t_{22})}
\right]t_{12}t_{23}.
\end{aligned}
\]
\end{ex}

\begin{theorem} \label{thm: Karcher mean 3x3}
Let $\w =(w_1,\ldots,w_m)\in (0,1)^m$ be weights such that $w_1+\cdots+w_m=1$.
Let $A_1,\ldots, A_m\in \T_3$. Then the $\w$-weighted Karcher mean is
	\[
		X =G(\w ; A_1,\ldots, A_m)= DX', 
\]
where   
\[
		X' =G(\w ; A_{(1)},\ldots, A_{(m)})
		=\begin{bmatrix}
			1&x_{12}' &x_{13}' \\ 0 &1 &x_{23}' \\ 0 &0 &1
		\end{bmatrix},
	\]
and the entries 
$x_{12}', x_{23}', x_{13}'$  are given by
	\[
	\begin{aligned}
		x_{12}' =
		\frac{1}{c_{12}}
		\sum_{i=1}^m  w_i\,
		& \log^{[1]} \bigl(a_{11}^{(i)},a_{22}^{(i)}\bigr)\, a_{12}^{(i)},
		\\
		x_{23}' =
		\frac{1}{c_{23}}\sum_{i=1}^m
		w_i\,
		& \log^{[1]} \bigl(a_{22}^{(i)},a_{33}^{(i)}\bigr)\, a_{23}^{(i)},
		\\
		x_{13}' =\frac{1}{c_{13}}
		\sum_{i=1}^{m} w_i
		&\Bigg[
		\log^{[1]} \bigl(a_{11}^{(i)},a_{33}^{(i)}\bigr)
		\big(a_{13}^{(i)} - x_{12}' a_{23}^{(i)} + x_{12}'x_{23}' a_{33}^{(i)}\big)
		\\
		&\qquad \qquad 
		+
		\log^{[2]} \bigl(a_{11}^{(i)},a_{22}^{(i)},a_{33}^{(i)}\bigr)
		\big(a_{12}^{(i)} - x_{12}' a_{22}^{(i)}\big)
		\big(a_{23}^{(i)} - x_{23}' a_{33}^{(i)}\big)
		\Bigg]. 
	\end{aligned}
	\]
\end{theorem}

\begin{ex}\label{thm: Log 4x4}
 Theorem \ref{thm:Tn-analytic} implies that
the logarithm of a  $4\times 4$ matrix
  $T=\mtx{t_{ij}} \in\T_4$ is
$$
	\log(T) =\mtx{f_{ij}}\in\t_4, 
$$
	in which
	$f_{ii}=\log(t_{ii})$ for $i=1,\ldots, 4,$ and 
	\[
	\begin{aligned}
		f_{12} &= \log^{[1]} (t_{11},t_{22}) \, t_{12},
		\qquad 
		f_{23} = \log^{[1]} (t_{22},t_{33}) \, t_{23},
		\qquad 
		f_{34} = \log^{[1]} (t_{33},t_{44}) \, t_{34},
		\\
		f_{13} &=
		\log^{[1]} (t_{11},t_{33})\,t_{13} + 
		\log^{[2]} (t_{11},t_{22},t_{33}) \, t_{12}t_{23}, 
		\\
		f_{24} &=
		\log^{[1]} (t_{22},t_{44})\,t_{24} + 
		\log^{[2]} (t_{22},t_{33},t_{44}) \, t_{23}t_{34}, 
		\\
		f_{14}
		&=
		\log^{[1]} (t_{11},t_{44})\,t_{14}
		+
		\log^{[2]} (t_{11},t_{22},t_{44})\,t_{12}t_{24}
		+
		\log^{[2]} (t_{11},t_{33},t_{44})\,t_{13}t_{34}
		+
		\log^{[3]} (t_{11},t_{22},t_{33},t_{44})\,t_{12}t_{23}t_{34}.
	\end{aligned}
	\]
\end{ex}

Theorem \ref{thm: Tn Karcher mean} for $n=4$ leads to the following result. 

\begin{theorem} \label{thm: Karcher mean 4x4}
	On \(\T_4\), the Karcher mean \(G(\w;A_1,\ldots,A_m)\) is given by
\[
X=G(\w;A_1,\ldots,A_m)=DX',
\]
where 
\[
X'=G(\w;A_{(1)},\ldots,A_{(m)}) =\mtx{x_{ij}'}
\]
is unipotent upper triangular such that
$X'[1:3]$ is determined by
Theorem \ref{thm: Karcher mean 3x3}
and 
	\[
	\begin{aligned}
		x_{34}' =
		\frac{1}{c_{34}}\sum_{i=1}^m
		w_i  \,
		& \log^{[1]} \bigl(a_{33}^{(i)},a_{44}^{(i)}\bigr)\, a_{34}^{(i)},
		\\
		x_{24}' =\frac{1}{c_{24}}
		\sum_{i=1}^{m} w_i & 
		\Bigg[
		\log^{[1]} \bigl(a_{22}^{(i)},a_{44}^{(i)}\bigr)
		\big(a_{24}^{(i)} - x_{23}' a_{34}^{(i)} + x_{23}'x_{34}' a_{44}^{(i)}\big) 
		\\
		&\qquad  
		+
		\log^{[2]} \bigl(a_{22}^{(i)},a_{33}^{(i)},a_{44}^{(i)}\bigr)
		\big(a_{23}^{(i)} - x_{23}' a_{33}^{(i)}\big)
		\big(a_{34}^{(i)} - x_{34}' a_{44}^{(i)}\big)
		\Bigg],
		\\
		x_{14}' =  \frac{1}{c_{14}}\sum_{i=1}^{m} w_i 
		&   
		\Bigg[
		\log^{[1]} \bigl(a^{(i)}_{11},a^{(i)}_{44}\bigr)
		\bigl(a^{(i)}_{14}
		- x'_{12}a^{(i)}_{24} 	- x'_{13}a^{(i)}_{34}
		\\
		&\qquad \qquad \qquad 
		+ x'_{12}x'_{23}a^{(i)}_{34}
		+ x'_{12}x'_{24}a^{(i)}_{44} 
		+ x'_{13}x'_{34}a^{(i)}_{44} 
		- x'_{12}x'_{23}x'_{34}a^{(i)}_{44}
		\bigr)
		\\
		&\qquad 
		+
		\log^{[2]} \bigl(a^{(i)}_{11},a^{(i)}_{22},a^{(i)}_{44}\bigr)
		\bigl(a^{(i)}_{12}-x'_{12}a^{(i)}_{22}\bigr)
		\bigl(a^{(i)}_{24}-x'_{23}a^{(i)}_{34}-x'_{24}a^{(i)}_{44}+x'_{23}x'_{34}a^{(i)}_{44}\bigr)
		\\
		&\qquad 
		+
		\log^{[2]} \bigl(a^{(i)}_{11},a^{(i)}_{33},a^{(i)}_{44}\bigr)
		\bigl(a^{(i)}_{13}-x'_{12}a^{(i)}_{23}-x'_{13}a^{(i)}_{33}+x'_{12}x'_{23}a^{(i)}_{33}\bigr)
		\bigl(a^{(i)}_{34}-x'_{34}a^{(i)}_{44}\bigr)
		\\
		&\qquad 
		+
		\log^{[3]} \bigl(a^{(i)}_{11},a^{(i)}_{22},a^{(i)}_{33},a^{(i)}_{44}\bigr)
		\bigl(a^{(i)}_{12}-x'_{12}a^{(i)}_{22}\bigr)
		\bigl(a^{(i)}_{23}-x'_{23}a^{(i)}_{33}\bigr)
		\bigl(a^{(i)}_{34}-x'_{34}a^{(i)}_{44}\bigr)
		\Bigg].
	\end{aligned}
	\]
\end{theorem}

\subsection{A Block Recursive Algorithm}

To derive a symbolic expression for the Karcher mean \(X=DX'\) from Theorem \ref{thm: Tn Karcher mean}, one first computes \(D\) and then determines the entries \((X')_{1q}\) for \(2\le q\le n\); the remaining entries are obtained from analogous formulas. This requires \(n-1\) recursive stages. For numerical computation, however, the recursion must be carried out for each of the \(\frac{n(n-1)}{2}\) strictly upper-triangular entries of \(X'\). A more efficient approach is therefore to employ recursive block-matrix operations. The formulas simplify further when the diagonal entries of each \(A_{(i)}\) are pairwise distinct.

\begin{lemma}\label{thm: log-block}
	Let
	\[
	T=\begin{bmatrix}
		T_{11} & T_{12}\\
		0 & t_{nn}
	\end{bmatrix}\in \T_n,
	\qquad T_{11}=T[1:n-1].
	\]
	Then
	\[
	\log(T)=
	\begin{bmatrix}
		\log(T_{11}) & \log^{[1]} (T_{11},\, t_{nn}I_{n-1})\,T_{12}\\
		0 & \log(t_{nn})
	\end{bmatrix},
	\]
	where $\log^{[1]} (T_{11}, t_{nn}I_{n-1})$ denotes the analytic functional calculus of the $\log^{[1]} $ function applied to the commuting pair $(T_{11}, t_{nn}I_{n-1})\in (\T_{n-1})^2$. In particular, if no diagonal entry of $T_{11}$ equals $t_{nn}$, then
	\[
	\log^{[1]} (T_{11}, t_{nn}I_{n-1})
	=
	(\log(T_{11})-\log(t_{nn})I_{n-1})(T_{11}-t_{nn}I_{n-1})^{-1}.
	\]
\end{lemma}

	\begin{proof}
		Apply the Parlett recurrence \eqref{Parlett recurrence block} to the $(1,2)$-block with $f(x)=\log x$. This yields
		\[
		(T_{11}-t_{nn}I_{n-1})(\log T)_{12}
		=
		(\log(T_{11})-\log(t_{nn})I_{n-1})\,T_{12}.
		\]
		If no diagonal entry of $T_{11}$ equals $t_{nn}$, then $T_{11}-t_{nn}I_{n-1}$ is invertible, and hence
		\[
		(\log T)_{12}
		=
		(T_{11}-t_{nn}I_{n-1})^{-1}
		(\log(T_{11})-\log(t_{nn})I_{n-1})\,T_{12}.
		\]
		The general case follows by analytic continuation of the divided difference function $ \log^{[1]} $. This completes the proof.
	\end{proof}

The following theorem shows how to find the Karcher mean on \(\T_n\) using \((n-1)\) steps of recursive block-matrix iterations. 

\begin{theorem}\label{thm:Tn-Karcher-mean-block}
	Let
	$
	X=G(\w;A_1,\ldots,A_m)
	$
	be the \(\w\)-weighted Karcher mean on \(\T_n\). Set
	\[
	D=\diag(A_1^{w_1}\cdots A_m^{w_m}),\qquad
	A_{(i)}=D^{-1}A_i.
	\]
	Then \(X=DX'\), where
	$
	X' =G(\w;A_{(1)},\ldots,A_{(m)})
	$
	is unipotent upper triangular. Moreover, \(X'\) can be computed recursively as follows. For \(1\le \ell\le n\), define
	\[
	X'_{\ell}:=X'[1:\ell],\qquad
	A_{(i,\ell)}:=A_{(i)}[1:\ell],\quad i=1,\ldots,m.
	\]
	
	\begin{enumerate}
		\item \(\diag(X')=I_n\).
		
		\item Suppose that \(X'_{\ell-1}\) has been determined for some \(1<\ell\le n\). Then
		\[
		X'_{\ell}
		=
		\begin{bmatrix}
			X'_{\ell-1} & \y'_{\ell-1}\\
			0 & 1
		\end{bmatrix},
		\]
		where \(\y'_{\ell-1}\in\R^{\ell-1}\) is determined as follows. Write
		\[
		A_{(i,\ell)}
		=
		\begin{bmatrix}
			A_{(i,\ell-1)} & \b_{(i,\ell-1)}\\
			0 & a_{\ell\ell}^{(i)}
		\end{bmatrix},
		\]
		and define
		\[
		G_{i,\ell-1}
		=
		\log^{[1]} \!\left(
		(X'_{\ell-1})^{-1}A_{(i,\ell-1)},\;
		a_{\ell\ell}^{(i)}I_{\ell-1}
		\right).
		\]
		Then we have 
\begin{equation}
    		\y'_{\ell-1}
		=
		X'_{\ell-1}
		\left(
		\sum_{i=1}^m w_i a_{\ell\ell}^{(i)}\, G_{i,\ell-1}
		\right)^{-1}
		\left[
		\sum_{i=1}^m
		w_i\, G_{i,\ell-1}(X'_{\ell-1})^{-1}\b_{(i,\ell-1)}
		\right].
\end{equation}
	\end{enumerate}
\end{theorem}

\begin{proof}
	By Section 2.1,
	\[
	X'_{\ell}=G(\w;A_{(1,\ell)},\ldots,A_{(m,\ell)}),
	\]
	so the Karcher equation gives
	\[
	0=\sum_{i=1}^{m} w_i\log\bigl((X'_{\ell})^{-1} A_{(i,\ell)} \bigr).
	\]
	Using
	\[
	(X'_{\ell})^{-1}
	=
	\begin{bmatrix}
		(X'_{\ell-1})^{-1} & -(X'_{\ell-1})^{-1}\y'_{\ell-1}\\
		0 & 1
	\end{bmatrix},
	\]
	we obtain
	\[
	\begin{aligned}
		0
		&=\sum_{i=1}^{m} w_i\log \left(
		\begin{bmatrix}
			(X'_{\ell-1})^{-1} A_{(i,\ell-1)} &
			(X'_{\ell-1})^{-1} \b_{(i,\ell-1)} - a^{(i)}_{\ell\ell}(X'_{\ell-1})^{-1}\y'_{\ell-1}\\
			0 & a^{(i)}_{\ell\ell}
		\end{bmatrix}
		\right).
	\end{aligned}
	\]
	Evaluating the \((1,2)\)-block using Lemma \ref{thm: log-block} yields
	\[
	\begin{aligned}
		0
		&=\sum_{i=1}^{m} w_i\,
		\log^{[1]} \bigl((X'_{\ell-1})^{-1} A_{(i,\ell-1)},\, a^{(i)}_{\ell\ell}I_{\ell-1}\bigr) \, 
		\left[(X'_{\ell-1})^{-1} \b_{(i,\ell-1)} - a^{(i)}_{\ell\ell}(X'_{\ell-1})^{-1}\y'_{\ell-1}\right]
		\\
		&=
		\sum_{i=1}^{m} w_i\, G_{i,\ell-1}
		\left[(X'_{\ell-1})^{-1} \b_{(i,\ell-1)} - a^{(i)}_{\ell\ell}(X'_{\ell-1})^{-1}\y'_{\ell-1}\right].
	\end{aligned}
	\]
	Hence
	\[
	\sum_{i=1}^{m} w_i G_{i,\ell-1} (X'_{\ell-1})^{-1} \b_{(i,\ell-1)}
	=
	\sum_{i=1}^{m} w_i a^{(i)}_{\ell\ell} G_{i,\ell-1} (X'_{\ell-1})^{-1}\y'_{\ell-1},
	\]
	which gives
	\[
	\y'_{\ell-1}
	=
	X'_{\ell-1}
	\left(\sum_{i=1}^{m} w_i a^{(i)}_{\ell\ell} G_{i,\ell-1}\right)^{-1}
	\left[\sum_{i=1}^{m} w_i G_{i,\ell-1} (X'_{\ell-1})^{-1} \b_{(i,\ell-1)}\right].
	\qedhere
	\]
\end{proof}


\begin{thebibliography}{99}

\bibitem{dB05}
C.\ de Boor,  Divided Differences, Surv. Approx. Theory 1 (2005), 46--69.



\bibitem{Hu26}
H.\ Huang, The Karcher mean on $2\times 2$ triangular matrices with positive diagonal entries, manuscript.

\bibitem{Law25}
 J.\ Lawson,  The weighted exponential (Karcher) mean on Lie groups. Acta Sci. Math. (Szeged) 92, 323–341 (2026).
 \href{https://doi.org/10.1007/s44146-025-00187-5}{https://doi.org/10.1007/s44146-025-00187-5}

\bibitem{Law24}
J.\ Lawson, Weighted Karcher means on unipotent Lie groups. Adv. Oper. Theory 9, 27 (2024). \href{https://doi.org/10.1007/s43036-024-00326-9}{https://doi.org/10.1007/s43036-024-00326-9}


\bibitem{Par76}
 B.\ Parlett,  A recurrence among the elements of functions of triangular matrices, Linear Algebra
Appl., 14 (1976), 117-121.

\bibitem{PA12}
 X.\ Pennec and V.\ Arsigny, Exponential barycenters of the canonical Cartan connection and
invariant means on Lie Groups, Frederic Barbaresco and Amit Mishra and Frank Nielsen, Matrix
Information Geometry, Springer (2012), 123--168.

\bibitem{PL20}
 X.\ Pennec and  M.\ Lorenzi, Beyond Riemannian geometry: the affine connection setting for 
 transformation groups, Riemannian Geometric Statistics in Medical Image Analysis 2020,
 169-229.  


\end{thebibliography}

\end{document}